\documentclass{article}
\let\oldb\b
\let\oldl\l
\usepackage{notes,hyperref}
\let\l\lambda
\let\wtilde\tilde
\let\bar\overline
\def\O{\mathcal O}
\def\sing{{\mathrm{sing}}}
\let\b\beta

\theoremstyle{theorem}
\renewtheorem{thm}{Theorem}
\renewtheorem{lem}[thm]{Lemma}
\renewtheorem{cor}[thm]{Corollary}
\renewtheorem{prop}[thm]{Proposition}
\renewtheorem{conj}[thm]{Conjecture}
\numberwithin{thm}{section}
\theoremstyle{definition}
\renewtheorem{dfn}[thm]{Definition}
\renewtheorem{exa}[thm]{Example}
\renewtheorem{note}[thm]{Note}
\renewtheorem{nota}[thm]{Notation}
\renewtheorem{quest}[thm]{Question}
\renewtheorem{rem}[thm]{Remark}

\title{Homogeneous coordinate rings\\as direct summands of regular rings}
\author{Devlin Mallory\thanks{This material is based upon work supported by the National Science Foundation under Grant No.~\#1840190.}}
\begin{document}
\maketitle

\begin{abstract}
We study the question of when a ring can be realized as a direct summand of a regular ring by examining the case of homogeneous coordinate rings.
We present very strong obstacles to expressing a graded ring with isolated singularity as a \emph{finite} graded direct summand.
For several classes of examples (del Pezzo surfaces, hypersurfaces), we give a complete classification of which coordinate rings can be expressed as direct summands (not necessarily finite), and in doing so answer a question of Hara about the FFRT property of the quintic del Pezzo. 
We also examine what happens in the case where the ring does not have isolated singularities, through topological arguments: as an example, we give a classification of which coordinate rings of singular cubic surfaces can be written as finite direct summands of regular rings.
\end{abstract}

\section{Introduction}

A ring $R$ is called a direct summand of a regular ring if there is a regular ring $S$, an inclusion $R\hookrightarrow S$, and an $R$-module splitting $S\to R$ of the inclusion.
Direct summands of regular rings form a well-studied class of rings, and one for which many hard problems are more tractable. For example, direct summands of regular rings (or invariant rings that may not be direct summands) have been studied in contexts such as local cohomology \cite{NB,HNB}, differential operators \cite{Smith,AMHNB,MHJNB}, and Frobenius splittings \cite{SVdB,RSB}.

Because these rings have so many nice properties, one should expect them to be somewhat rare. 
However, it is quite quite hard to formulate and prove such statements: Given a ring $R$, how can one ensure that there is no possible regular ring $S$ and a split inclusion $R\hookrightarrow S$?
There are some ``easy'' properties such a ring must satisfy: $R$ must be
 a normal domain,
 Cohen--Macaulay \cite{HR},
 have rational singularities (in characteristic 0) \cite{Boutot}, and so forth. But there are many rings with these necessary properties that have no obvious reason to be a direct summand of a regular ring.

To show a ring is not a direct summand of a regular ring,
one can attempt to calculate the ``difficult'' invariants mentioned above (e.g., associated primes of local cohomology, differential operators, Frobenius splittings), and indeed to date this is perhaps the only known way to show that a ring is not a direct summand of a regular local ring (see, for example, \cite{NB}). However, computation of these invariants is often  not feasible, and so one would like to have other methods to show that a ring is not a direct summand.

In this paper, we examine the case of graded rings,  and in particular those graded rings arising as the homogeneous coordinate rings of smooth projective varieties, or those with mild singularities. Such rings provide a rich class of examples, demonstrating both a relative rarity of examples of direct summands and the subtlety of determining whether a ring is a direct summand.

In Section~\ref{finsum},
as a consequence of algebro-geometric arguments involving characterizations of projective space, we obtain the following restrictive criteria for the homogeneous coordinate ring of a smooth variety to be a \emph{finite} graded direct summand of a regular ring:

\begin{thm}
Let $R$ be a $\N$-graded $\C$-algebra, finitely generated in degree 1, with an isolated singularity. Then $R$ is not a finite graded direct summand of a polynomial ring unless $\Proj R\cong \P^n$.
\end{thm}

For certain classes of examples, we can go beyond the module-finite case and give a complete characterization of which coordinate rings are direct summands. For example, in Section~\ref{delPezzo} we obtain the following:

\begin{thm}
A complex (smooth) del Pezzo surface $X_d$ of degree $d$ has a homogeneous coordinate ring that is a direct summand of a regular ring if and only if $d\geq 5$.
\end{thm}

As a consequence of our proof, we answer a question posed by \cite[Section~3]{Hara2}:

\begin{thm}
The homogeneous coordinate ring of a degree-5 del Pezzo in characteristic $p>2$ is a direct summand of a ring with the FFRT property, and thus itself has FFRT.
\end{thm}

We can give a similar classification in the case of coordinate rings of hypersurfaces of degree $d\geq 3$ in $\P^n$; see Section~\ref{hypersurfaces}.

In Section~\ref{singcoh},
via an examination of the singular cohomology of del Pezzo surfaces with quotient singularities, we can also give a nearly complete classification of which singular cubic surfaces are direct summands of regular rings, and complete the classification by drawing on results of \cite{Gurjar}:

\begin{thm}
The homogeneous coordinate ring of a singular complex cubic surface is a finite graded direct summand of a regular ring if and only if the surface has quotient singularities of type $3A_2$.
\end{thm}

We believe that the techniques used in the proof of this theorem are novel in an algebraic context, and may be of more interest than the precise classification of which singular cubics are finite graded direct summands.

\subsection*{Acknowledgments}
We'd like to thank Anurag Singh, Eamon Quinlan--Gallego, Monica Lewis, Michael Perlman, Jakub Witaszek, Shravan Patankar, and Swaraj Pande for helpful conversations and suggestions regarding this work.
We'd also like to thank Mircea Musta\c{t}\u{a} and Eamon Quinlan--Gallego for thorough feedback on a draft of this article.


\section{Preliminaries}
We begin by reviewing some definitions, as well as establishing a few basic lemmas we will use in what follows.
We will then detail a number of different concepts and techniques we will make use of. We draw on a range of tools from both algebraic geometry and commutative algebra, so we have attempted to include at least a brief description of each.

\subsection{Direct summands, pure subrings, and splittings}

\begin{dfn}
We say a ring inclusion $R\hookrightarrow S$ is:
\begin{itemize}
\item a direct summand, if there is an $R$-linear splitting $\sigma:S\to R$.
\item pure, if $R\otimes_R M \to S\otimes_R M$ is injective for all $R$-modules $M$.
\end{itemize}
A ring $R$ is a direct  summand (respectively, pure subring) of a regular ring if there's a ring inclusion $R\hookrightarrow S$, with $S$ a regular ring, that is a direct summand (respectively, pure).

If $R\hookrightarrow S$ can be taken to be module-finite, we will say that $R$ is a module-finite direct summand of a regular ring, and likewise say $R$ is a graded direct summand if $R$ is homogeneous and $R\hookrightarrow S$ is a graded inclusion into a polynomial ring.
\end{dfn}

Being a direct summand implies being pure. As to the converse, it holds in the case where $S$ is a finite $R$-module, by the following more general theorem:

\begin{lem}{{\cite[Corollary~5.3]{HR}}}
$R\hookrightarrow S$ is pure if and only if $R$ is a direct summand of every finitely generated $R$-submodule of $S$.
\end{lem}

The two notions also coincide in the complete case; see \cite[Lemma~1.9]{CGM}. We will deal almost entirely with direct summands in this paper, but many of the results can be restated for pure subrings.

When we have a module-finite ring inclusion $R\hookrightarrow S$ and $S$ is normal, the existence of a splitting is guaranteed, at least in characteristic 0, by the standard well-known fact:

\begin{lem}
Let $R$ be a normal domain, and $S$ any domain, and say $R\hookrightarrow S$ is module-finite. As long as $\rank_R S$ has an inverse in $R$, there is an $R$-linear splitting $\sigma : S\to R$.
\end{lem}

Note that the assumption that $\rank_R S$ is invertible in $R$ is always true if $\Q\subset R$.

\begin{proof}
Let $\sigma_0 : \Frac S \to \Frac R$ be the trace map (i.e., $\sigma_0(s)$ is the trace of the $R$-linear operator ``multiplication-by-$s$'' on $S$).
Normality of $R$ guarantees that if $r\in R$, then $\sigma_0(s)$ is in $R$, rather than just $\Frac R$.
Thus, $\sigma_0$ descends to an $R$-linear map $S\to R$.
Then $\sigma_0(1)$ is trace of the identity map on the $R$-module $S$, so $\sigma_0(1)=\rank_R S$, and as long as $\rank _R S\neq 0$ in $R$  we can divide by it to obtain $\sigma:S\to R$ with $\sigma(1)=1$.
\end{proof}

Thus, in characteristic 0 (or in ``large'' characteristic $p$ compared to $\rank_R S$) just being a module-finite subring of a regular ring implies being a direct summand.

\begin{rem}
A Noetherian domain is called a \emph{splinter} if it is a direct summand of every module-finite extension.
The above lemma says that normal rings containing a field of characteristic 0 are splinters. 
The property of being a splinter is quite subtle in general: the direct summand conjecture, recently settled by  \cite{Andre}, is the assertion that all regular rings are splinters.
The focus of this article is in some sense the inverse question: not whether any module-finite extension of a regular ring splits, but which rings admit a regular ring as a module-finite extension that splits.
\end{rem}

As mentioned in the introduction, direct summands of regular rings satisfy a number of nice properties:
\begin{itemize}
\item They are normal and Cohen--Macaulay \cite{HR}.
\item They have rational singularities in characteristic 0 \cite{Boutot}.
\item They have strongly $F$-regular singularities in characteristic $p$ \cite[Theorem~2.2.(4)]{Smith}.
\item Their local cohomology modules have only finitely many associated primes \cite{NB}.
\item They are $D$-simple \cite[Proposition~3.1]{Smith}; see Section~\ref{diffops}, and admit Bernstein--Sato polynomials \cite{AMHNB,BSp}.
\end{itemize}

\begin{rem}

In general, being a direct summand of a (localization of a) polynomial ring is much stronger than being a direct summand of a regular local ring: for example, any regular local ring is tautologically a direct summand of a regular local ring, but unirationality is a necessary condition to be a subring of a localization of a polynomial ring. So, for example, the local ring of a point on an elliptic curve can't be a subring of a polynomial ring.
In the graded setting, however, we need not worry about this distinction, because of the following elementary observation:

\begin{lem}
Let $S$ be a graded $k$-algebra of dimension $n$, finitely generated over $R_0=k$. If $S$ is regular, then $S\cong k[x_1,\dots,x_n]$ as graded rings, where $\deg x_i=d_i$ for some positive integers $d_i$.
\end{lem}

\begin{proof}
Since $S$ is regular of dimension $n$ and the assumption that $R_0=k$ implies that the homogeneous ideal is maximal, there are homogeneous generators $x_1,\dots,x_n$ for the homogeneous maximal ideal. These clearly generate $S$ as a $k$-algebra as well, since $R_0$ is just $k$ itself. If there were an algebraic relation among the $x_i$ then $\dim S<n$, so we must have that $S$ is isomorphic to the polynomial ring $k[x_1,\dots,x_n]$.
\end{proof}
\end{rem}

We will also make use of the following:

\begin{lem}
Let $R\hookrightarrow S$ admit an $R$-linear splitting $\sigma_R:S\to R$, and say $S\hookrightarrow T$ admits an $S$-linear splitting $\sigma_S: T\to S$. Then $\sigma_R\circ \sigma_S : T\to R$ is an $R$-linear splitting of $R\hookrightarrow S\hookrightarrow T$.
\end{lem}

In particular, a direct summand of a direct summand of a regular ring is itself a direct summand of a regular ring.

\begin{proof}
\label{dirsum2}
Since $\sigma_S$ is $S$-linear, it is in particular $R$-linear, and the result is clear.
\end{proof}

We can descend splittings to quotients:

\begin{lem}
\label{splitquot}
Let $R\hookrightarrow S$, with $R$-linear map $\sigma : S\to R$. Let $I\subset R$ be an ideal with $IS\cap R=I$. Then $R/I\hookrightarrow S/IS$ admits an $R/I$-linear map $\bar\sigma:S/IS\to R/I$, defined by $\bar\sigma(\bar f)=\overline{\sigma(f)}$ for $\bar f \in S/IS$. If $\sigma$ is a splitting, so is $\overline \sigma$.
\end{lem}

The condition $IS\cap R=I$ is satisfied for pure ring maps.

\begin{proof}
First, note that the ring map $R/I\to S/IS$ is injective exactly when $IS\cap R=I$.
If $f=f'+g$ for $g\in IS$, then we can write $g=sg_0$ for $s\in S$, $g_0\in R$. Then $\sigma(f)=\sigma(f')+\sigma(sg_0)=
\sigma(f')+g_0\sigma(s)$. Thus, $\sigma(f)=\sigma(f') \mod I$, so $\bar \sigma$ is well-defined. The $R/I$-linearity of $\bar \sigma$ is clear from definition.
Note that if 
$\sigma(1)=1$
(i.e., $\sigma$ is a splitting), then the same is true for $\bar \sigma$.
\end{proof}

\subsection{Characteristic 0 singularities of module-finite direct summands}
We can obtain stronger conditions on the singularities of a \emph{module-finite} direct summand in characteristic 0.
The following statement is known to experts, but we have been unable to find a reference in the literature:

\begin{lem}
\label{klt}
Let $R$ be a  $\C$-algebra, essentially of finite type over $\C$, and assume $R$ is a module-finite subring of a regular ring $S$ such that $R$ is $\Q$-Gorenstein (e.g., if $S$ is a polynomial ring or regular local ring). Then $R$ has klt singularities.
\end{lem}

The condition on $R$ being $\Q$-Gorenstein is automatic if $S$ is a polynomial ring or a regular local ring, since then the class group of $R$ is torsion (see Corollary~\ref{torsClass}).

\begin{proof}
The proof is a standard reduction to characteristic $p$, so we sketch the main idea:
Assume $R\subset S$, and let $\sigma \in \Hom_R(S,R)$ be a splitting. Because $S$ is a finite $R$-module, $\Hom_R(S,R)$ is as well. There is thus a finitely generated $\Z$-algebra $A$, $A$-algebras $R_A\subset S_A$, and $\sigma_A \in \Hom_{R_A}(S_A,R_A)$ such that 
$R_A\otimes_A \C=R$, 
$S_A\otimes_A \C=S$, and
$\sigma_A\otimes_A \C=\sigma$, and an open dense set $U\subset \Spec A$ such that for all $p\in U$ we have an inclusion $$R_p:=R_A\otimes_A A_p/pA_p \hookrightarrow S_p:=
S_A\otimes_A A_p/pA_p $$
and a splitting $\sigma_p:= \sigma_A\otimes_A A_p/pA_p : S_p\to R_p$. We can moreover assume that $S_p$ is regular for all $p\in U$. The residue fields $A_p/pA_P$ are finite fields, and because $R_p$ is a direct summand of $S_p$, $R_p$ must be strongly $F$-regular by \cite[Theorem~2.2.(4)]{Smith}. 
Thus, we have that $R$ is a $\Q$-Gorenstein ring of strongly $F$-regular type, and thus klt by \cite[Theorem~3.3]{HW}.
\end{proof}

\subsection{Class groups}
Let $R$ be a normal domain. In this section, we briefly recall the class group of $R$, and its behavior under finite ring extensions.

\begin{dfn}
The class group of $R$, denoted $\Cl(R)$, is the quotient of the free abelian group generated by height-1 prime ideals of $R$, modulo the principal divisors.
\end{dfn}

$\Cl(R)=0$ if and only if $R$ is a UFD. In particular, regular local rings are UFDs and thus have trivial class groups.

If $R\hookrightarrow S$ is a module-finite extension of normal domains, there are maps $i:\Cl(R)\to \Cl(S)$ and $j:\Cl(S)\to \Cl(R)$, defined as follows:
\begin{itemize}
\item $i: \Cl(R)\to \Cl(S)$: for $q$ a height-1 prime in $R$, let $Q_1,\dots,Q_r$ be the primes of $S$ lying over $q$, and note that $\Ht Q_i=1$. Let $R_q\to S_{Q_i}$ be the localized ring extension, and let $e_{Q_i} =  \l(S_{Q_i}/qS_{Q_i})$ be the ramification index at $Q_i$.
 Then, we define $$i([q]) := \sum e_{Q_i}[Q_i]\in \Cl(S).$$
\item $j:\Cl(S)\to \Cl(R)$: let $Q$ be a height-1 prime of $S$, and let $q=Q\cap R$. Note that $\Ht q=1$. Write $k(Q)=S_Q/QS_Q$ for the residue field of $Q$, and likewise for $k(q)$, and set 
$$
j([Q])= [k(Q):k(q)] [q].
$$
\end{itemize}

With these definitions, we then have (see, e.g., \cite[Section~6]{Samuel}):

\begin{prop}
Let $[\Frac S:\Frac R] = N$. Then
$j\circ i :\Cl(R)\to \Cl(R)$ is multiplication by $ N$.
\end{prop}

This immediately implies the following.

\begin{cor}
\label{torsClass}
If $R$ is a module-finite subring of a UFD $S$, then $\Cl(R)$ is torsion.
\end{cor}

In our context, $S$ will be a polynomial ring or a regular local ring and thus automatically a UFD.

\subsection{Linearly reductive groups}

One source of direct summands of polynomial rings, especially in characteristic 0, is linearly reductive groups.

\begin{dfn}
An algebraic group $G$ is called linearly reductive if every finite-dimensional $G$-representation splits into a direct summand of \emph{irreducible} $G$-representations. 
\end{dfn}

\begin{rem}

If $G$ acts on an $n$-dimensional vector space $V$,  there is an induced action of $G$  on the polynomial ring $\Sym V=k[x_1,\dots,x_n]=S$. If $G$ is linearly reductive, the invariant subring $R:=S^G$ is a direct summand of $S$.
That is, rings of invariants of linearly reductive groups are direct summands of regular rings. 
\end{rem}

In characteristic 0, finite groups are linearly reductive, as are tori and the familiar matrix groups (e.g., $\GL_n,\SL_n,O_N$, and so on). In what follows, we will use only that finite groups, tori, and $\GL_n$ are linearly reductive.

In characteristic $p$, finite groups are linearly reductive as long as the characteristic $p$ doesn't divide the order of the group. Tori are also always linearly reductive in characteristic $p$.
However, these are essentially the only examples: by \cite{Nagata} linearly reductive groups in positive characteristic are extensions of tori by finite groups with order not divisible by the characteristic.

\begin{rem}
In positive characteristic, many familiar rings (e.g., determinantal rings, Grassmannians, etc.) are still rings of invariants, exactly as they are in the characteristic-0 setting. However, what is missing in this case is the existence of a splitting (i.e., in the context of invariant theory, the Reynolds operator). Without the existence of the splitting guaranteed by linear reductivity, it is quite hard to determine whether these invariant rings are direct sums of polynomial rings.

For example, forthcoming work of\cite{HJPS} examines the invariant rings of classical determinantal groups (e.g., $\GL_n$, the symplectic group, etc.) over a field of positive characteristic. They show that the inclusion of this invariant ring essentially never splits (or is even pure).
This does not say that there is no \emph{other} embedding of the invariant ring making it a direct summand, but just that the most ``natural'' choice does not work.

Similarly, \cite{JS} showed that these classical invariant rings over $\Z$ are essentially never direct summands of \emph{any} regular ring, through a study of their differential operators and reduction modulo~$p$.
\end{rem}

To emphasize the ubiquity of rings of invariants among direct summands, we mention that the following question is (to our knowledge) open:

\begin{quest}
In characteristic 0, is there a direct summand of a regular ring $R\hookrightarrow S$ that cannot be written as  the ring of invariants of a linearly reductive group acting on a polynomial ring $T$? 
\end{quest}

Even in the case where the direct summand is assumed to be finite, it is unclear how to produce a group action yielding the ring as a ring of invariants, or how to show that no such action can exist.

\subsection{Differential operators}
\label{diffops}
We briefly recall the notion of differential operators and their behavior under direct summands.

\begin{dfn}
Let $R$ be a $k$-algebra. We define $D^m_{R/k}\subset \End_k(R)$, the $k$-linear differential operators of order $m$ inductively as follows:
\begin{itemize}
\item $D^0_{R/k} =\Hom_R(R,R)\cong R$, thought of as multiplication by $R$. 
\item  $\delta \in \End_k(R)$ is in $D^m_{R/k}$ if $[\delta, r] \in D^{m-1}_{R/k}$ for any $r\in D^0_{R/k}$.
\end{itemize}
We write $D_{R/k}=\bigcup D^m_{R/k}$. $D_{R/k}$ is a noncommutative ring, and $R$ is a left $D_{R/k}$-module.
\end{dfn}

When $R$ is not a smooth $k$-algebra, the $D_{R/k}$-module structure of $R$ can be quite complicated. We will in particular make use of the notion of $D$-simplicity:

\begin{dfn}
A $k$-algebra $R$ is called $D$-simple if $R$ is a simple $D_{R/k}$-module.
\end{dfn}

Polynomial rings (and, more generally, regular rings) are $D$-simple.

Although an arbitrary singular ring $R$ may not be $D$-simple, the situation is much better for direct summands of regular rings. The following lemma and corollary are due to 
\cite[Proposition~3.1]{Smith}:

\begin{lem}
Let $R\to S$ be an inclusion of $k$-algebras, with $R$-linear splitting $\sigma : S\to R$.
For any differential operator $\delta$ on $S$, $\sigma \circ \delta\res R$ is a differential operator on $R$. 
\end{lem}

If $S$ is $D$-simple, then any element $s\in S$ can be sent to 1 by a differential operator,  and we can use thus send any element of $R$ to $1$, yielding: 

\begin{cor}
\label{dsimpsum}
If $R$ is a direct summand of a $D$-simple ring, $R$ is $D$-simple.
\end{cor}

\subsection{Finite $F$-representation type}
Though most of the paper is concerned with the characteristic 0 setting, in Section~\ref{delPezzo} our results will have ramifications in positive characteristic. We thus recall the necessary notion of finite $F$-representation type (or FFRT). 

\begin{dfn}
Let $R$ be a ring of characteristic $p>0$, and assume that $R$ is $F$-finite (i.e. the Frobenius morphism is finite). For $e>0$, we denote $F^e$ the $e$-th iterated Frobenius, and write $F^e_* R$ for the $R$-module obtained by restriction of scalars along $F^e:R\to R$. 
We say that $R$ has finite $F$-representation type (or FFRT) if there are finitely many finitely generated $R$-modules $M_1,\dots,M_N$ such that for any $e$ we can write
$$
F_*^e R \cong M_1^{\oplus a_{e,1}}\oplus \dots\oplus
M_N^{\oplus a_{e,N}}
$$
(as $R$-modules)
for some nonnegative $a_{e,i}$. 
That is, $R$ has FFRT if there are only finitely many modules occurring in an irreducible $R$-module decompositions of each $F_*^e R$.

We say that a graded ring $R$ has \emph{graded} FFRT if there are
finitely many finitely generated $\Q$-graded $R$-modules $M_1,\dots,M_N$ such that for any $e$ we can write
$$
F_*^e R \cong 
\bigoplus_{i=1}^N
\bigoplus_{j=1,\dots,n_{e,i}}
M_i^{a_{e,i,j}}(\b_{e,i,j})
$$
for some  $\b_{e,i,j}\in \Q$ and nonnegative $a_{e,i,j}$, where $M(\b)$ denotes the shift of the $\Q$-graded module $M$ by $\b$.
That is, only finitely many graded modules occur, up to shifting the grading, in 
an irreducible decomposition of the $F_*^e R$.

\end{dfn}

It is quite hard to calculate whether a ring satisfies FFRT in general, and FFRT is a very strong property. For example, a ring $R$ with FFRT will have only finitely many associated primes of the local cohomologies $H^i_I(R)$, for any ideal $I\subset R$ (see \cite{TT,HNB,DQ}).

The FFRT property behaves well with respect to direct summands, at least in the graded case:

\begin{prop}[{\cite[Proposition~3.1.6]{SVdB}}]
Let $R\hookrightarrow S$ be an inclusion of $F$-finite graded rings, with $R_0=S_0=k$ a field of characteristic $p$. Assume that $R\hookrightarrow S$ splits (or even if $F_*^e R \hookrightarrow F_*^eS$ splits for some $e$). Then $S$ having graded FFRT implies that $R$ does also.
\end{prop}

In particular, this does not require $R\hookrightarrow S$ to be module-finite, which will be important in our use in what follows.

\subsection{Remmert--Van de Ven-type theorems}
\label{RVdV}
The following theorems will be used in the study of finite direct summands; they are a strong obstacle for a smooth variety $X$ to admit a surjection from some projective space.

\begin{thm}[{\cite{Lazarsfeld}}]
\label{Lazarsfeld}
 Let $X$ be a smooth complex projective variety. If $\P^n\to X$ is a surjection, then $X\cong \P^m$
for some $m\leq n$.
\end{thm}

This result, which proved a conjecture of Remmert--Van de Ven,
relies on Mori's result characterizing projective space via ampleness of the tangent bundle and its consequences for the deformation theory of curves.

From this result, 
we then immediately obtain:

\begin{cor}
Let $X$ be a smooth complex projective variety and $\P(a_0,\dots,a_n) \to X$ a surjection.
Then $X\cong \P^m$ for some $m\leq n$.
\end{cor}

This follows by noting that $\P(a_0,\dots,a_n)$ is itself the surjective image of $\P^n$: it is the quotient of $\P^n$ by the finite group $\mu_{a_0}\times \dots\times \mu_{a_n}$.


\begin{rem}
In fact, Theorem~\ref{Lazarsfeld} remains true in positive characteristic, if one adds the condition of separability to the morphism $\P^n\to X$; see \cite[Section~V.3]{Kollar}.
\end{rem}

\begin{exa}
We note that the assumption on smoothness of $X$ is crucial. For example, the hypersurface $V(x_0^n-x_1\cdots x_n)\subset\P^n$ is the image of $\P^{n-1}$ under the finite morphism $\P^{n-1}\to \P^n$, $[a_0,\dots,a_{n-1}]\to [a_0\cdots a_{n-1},a_0^{n},\dots, a_{n-1}^n]$.
One can check that in this case, the finite morphism is in fact a quotient of $\P^{n-1}$ by the action of a finite group.
We do not know of an example where there is a surjection $\P^n\to X$, with $X$ not isomorphic to $\P^n$, that is not the quotient by a finite group.
\end{exa}

\subsection{Cox rings}
\label{cox}

For a projective variety $X$, the Cox ring of $X$ is a sort of ``universal coordinate ring'', which encodes the information of all homogeneous coordinate rings of $X$.

\begin{dfn}
Let $X$ be a factorial projective variety, and assume that $\Pic(X)$ is a free abelian group of rank $r$.
Let $D_1,\dots,D_r$ be a basis for $\Pic(X)$.
The Cox ring of $X$ (with respect to $D_1,\dots,D_r$) is the $\Z^r$-graded ring 
$$
R(X,D_1,\dots,D_r):=\bigoplus_{(a_1,\dots,a_r)\in \Z^r} H^0(X,\O_X(a_1D_1+\dots+a_r D_r)),
$$
with multiplication given by the natural multiplication of global sections 
$$
\displaylines{
H^0(X,\O_X(a_1D_1+\dots+a_r D_r))\otimes 
H^0(X,\O_X(m'_1D_1+\dots+m'_r D_r)) 
\hfill\cr\hfill
\to 
H^0(X,\O_X((a_1+m'_1)D_1+\dots+(a_r+m'_r) D_r)) .
}
$$
\end{dfn}
Different choices of basis yield isomorphic Cox rings, but the isomorphisms are not canonical. Nonetheless, we will suppress dependence on the choice of basis and write just $R(X)$ for the Cox ring of $X$.
The Cox ring can be defined when $\Cl(X)$ has torsion or $X$ is not factorial, but in our situation the above condition will be satisfied and so we avoid the more complicated issues this involves.

\begin{rem}
Let $D$ be a divisor, and say $D$ is linearly equivalent to $b_1D_1+\dots+b_r D_r$. We then recover the section ring $\bigoplus_m H^0(X,\O_X(mD))$ as 
$$
\bigoplus_m H^0(X,\O_X(mb_1D_1+\dots+mb_rD_r)) \subset R(X),
$$
i.e., as a ``Veronese subring'' of $R(X)$.
When $D$ is very ample and defines a projectively normal embedding, then, we get the homogeneous coordinate ring corresponding to this embedding as a Veronese subring of $R(X)$.
\end{rem}

We can rephrase this graded subring as the invariants of a torus action.
The $\Z^r$-grading gives a torus action of $\G_m^r$ on $R(X)$, with 
$$(\l_1,\dots,\l_r) \cdot s = 
\l_1^{a_1}
\cdots
\l_r^{a_r} s \quad\text{ for }\quad s\in H^0(X,\O_X(a_1D_1+\dots+a_r D_r)).
$$
Given a choice of divisor $D\sim b_1D_1+\dots+b_r D_r$,
we can extend this to a torus action on $R(X)[t]$, where 
$$ 
(\l_1,\dots,\l_r) \cdot t = \l_1^{-b_1}\cdots \l_r^{-b_r} t.
$$
Note that the invariants of this action are precisely 
the elements
$$
s  t^m, \quad s\in H^0(X,\O_X(mD)),
$$
and so the invariant ring is 
$$
(R[t])^{\G_m^r} \cong \bigoplus_m H^0(X,\O_X(mD)).
$$

\begin{rem}
The above is more commonly phrased as follows:
If $D$ is ample, the calculation above shows 
 that $\Proj(R[t])^{\G_m^r}=X$, and thus $X$ is a GIT quotient of $\Spec R[t]$. Different choices of ample divisor $D$ produce different GIT quotients yielding $X$.
Since we will be concerned directly with the homogeneous coordinate ring of $X$, rather than its Proj, we will need the full details of the above calculation.
\end{rem}

\section{Finite direct summands}
\label{finsum}

In this section, we consider 
algebro-geometric obstructions for a homogeneous $\C$-algebra $R$ to be a \emph{finite} graded direct summand.  In fact, we rule out even having a graded module-finite inclusion of $R$ into a polynomial ring.
(This is in fact the only obstruction to a normal ring being a module-finite summand of a polynomial ring: as mentioned previously, given any inclusion, the trace map on fraction fields descends to a splitting, at least for normal rings.)

\begin{thm}
\label{finite}
Let $R$ be a $\N$-graded $\C$-algebra, finitely generated in degree 1, with an isolated singularity. Then $R$ is not a finite graded direct summand of a polynomial ring, unless $\Proj R\cong \P^n$.
\end{thm}

The assumptions in the theorem guarantee that $R$ is of the form $S(X,L) =\bigoplus H^0(X,\O_X(mL))$ for $X$ a smooth projective variety.
The key is that $\Proj R$ is almost never $\P^n$, except for fairly well-understood exceptions. For example:

\begin{cor}
Homogeneous coordinate rings of
smooth (nondegenerate) complete intersections in $\P^N$ are not finite graded direct summands of regular rings, except for degree-2 curves in $\P^2$.
\end{cor}

The exceptional case is just the Veronese ring $\C[x,y,z]/(xz-y^2)\cong \C[u^2,uv,v^2]$, which is the Homogeneous coordinate ring of $\P^1$ as a plane conic.

\begin{rem}
If $R$ is generated in degree 1 and $\Proj R=\P^n$, and if $R$ is normal (which is necessary to be a direct summand of a regular ring), then $R$ must be a Veronese subring of a standard-graded polynomial ring. ($R$ is a homogeneous coordinate ring of $\P^n$ under some complete embedding.)
\end{rem}

\begin{rem}
We can weaken the assumption on generation of $R$ in degree 1. If $R$ is not generated in degree 1, choose $d$ such that the Veronese subring $R_{(d)}:= \bigoplus_m R_{md}$ is generated in degree 1. $R_{(d)}$ is a direct summand of $R$, and so if $R_{(d)}$ is not a direct summand of a polynomial ring $R$ cannot be either. However, note that $R_{(d)}$ may not have isolated singularities, even if $R$ is regular. 
For example, consider
$\C[x,y,z,w]/(xz-y^2)$, which has non-isolated singularities. But it is also
the 2-Veronese subring of the graded polynomial ring $\C[u,v,r,s]$, with $\deg u =\deg v=\deg r =1$ and $\deg s=2$.

So, the assumptions in the theorem can be weakened to ``$R$ has a Veronese subring generated in degree-1 with isolated singularities''.
\end{rem}

\begin{proof}[Proof of Theorem~\ref{finite}]
Assume that $R\hookrightarrow \C[x_0,\dots,x_n]$ is module-finite. 
Let $X=\Proj R$. 
By assumption, $X$ is a smooth projective variety.

Since $R\subset \C[x_0,\dots,x_n]=S$ is a module-finite  inclusion of graded rings, the homogeneous maximal ideal of $S$ must contract to the homogeneous maximal ideal of $R$. There is thus an induced surjection of projective varieties $\P^n \to \Proj R=X$ (if all generators of $R$ have the same degree in $S$) or $\P(a_0,\dots,a_n)\to X$ if the $i$-th generators has degree $a_i$.
Because $X$ is smooth, we can then apply the theorems of \cite{Lazarsfeld} (in the case $\P^n \to X$) or the corollary we stated  in Section~\ref{RVdV} (in the case $\P(a_0,\dots,a_n)$) to conclude that $\Proj R\cong \P^n$.
%
%
%
%
\end{proof}

\begin{rem}
We can also remove the characteristic-0 assumption, if we weaken the conclusion: we can only say that that there is no module-finite graded map $R\hookrightarrow S$ that is also separable. This follows since Theorem~\ref{Lazarsfeld} is true in positive characteristic if the map $\P^n \to X$ is assumed to be separable.
We do not know any examples of a characteristic-$p$ ring $R$ that is a direct summand of a regular ring $S$, but with the inclusion $R\hookrightarrow S$ necessarily inseparable.
As pointed out by Eamon Quinlan-Gallego,
a potential source of such examples is given by degree-$p$ Veronese subrings in characteristic $p$: for example, the obvious inclusion $R=\F_2[u^2,uv,v^2]\subset \F_2[u,v]$ is not separable (but it's possible there is some other inclusion of $R$ into a polynomial ring that does split).
\end{rem}

\begin{exa}
Again, we note that the hypothesis of smoothness is crucial.
As discussed previously, the hypersurface $V(x_0^n-x_1\cdots x_n)\subset\P^n$ is  the image of $\P^{n-1}$ under a finite morphism. The corresponding module-finite inclusion of coordinate rings 
$$
\C[x_0,\dots,x_n]/(x_0^n-x_1\cdots x_n) \hookrightarrow \C[a_0,\dots,a_{n-1}],\quad
x_0\mapsto a_0\cdots a_{n-1}, x_1\mapsto a_{0}^{n-1},\dots, x_n\mapsto a_{n-1}^{n-1}
$$
is easily seen to split, but $\Proj \C[x_0,\dots,x_n]/(x_0^n-x_1\cdots x_n)$ is not isomorphic to $\P^{n-1}$.
\end{exa}

\begin{rem}
It would be interesting to find non-toric examples of (necessarily singular) hypersurfaces or complete intersections in $\P^n$ whose homogeneous coordinate rings are finite direct summands of regular rings.
\end{rem}

\section{Quadric hypersurfaces}
\label{quad}

First, we recall the situation for quadrics in $\leq 4$ variables over an algebraically closed field of characteristic 0.
Write $Q_n$ for the quadric in $n$ variables (since we work over an algebraically closed field of characteristic 2, there is a unique such quadric up to isomorphism).
We need consider only the case of quadrics with isolated singularity at the origin (i.e., full-rank quadrics), since after a change of coordinates a quadric of lower rank is just a cone over a full-rank quadric, and so it suffices to consider the latter (as adjoining variables clearly doesn't affect being a direct summand of a regular ring).

\begin{itemize}
\item A quadric in $\leq 2$ variables is not a domain, so it cannot be a subring of any domain.
\item $k[x,y,z]/(xz-y^2) \cong k[u^2,uv,v^2]$, so $Q_3$ is a \emph{finite} direct summand of a polynomial ring.  $\Proj Q_3\cong \P^1$, so there is no contradiction with the results of the preceding section. Note also that this is independent of the characteristic 0 assumption.
\item $ Q_4=\C[x,y,z,w]/(xw-yz)$ is a toric ring, and so $Q_4$ is a graded direct summand of a polynomial ring. Explicitly, $ \C[x,y,z,w]/(xw-yz)\subset k[u,v,s,t]$ via $x\mapsto us, y\mapsto ut, z\mapsto vs, w\mapsto vt$. 
This is also independent of the characteristic 0 assumption.
However, since Proj of this ring is $\P^1\times \P^1$, it cannot be a \emph{finite} graded direct summand of a polynomial ring, by the above theorem. In fact, there is a more elementary argument showing that $Q_n$ cannot be a finite direct summand, even without the graded condition: $\Cl Q_3\cong \Z$, which is not torsion.
\end{itemize}

Our theorem in the preceding section implies the following for quadric hypersurfaces:

\begin{cor}
The homogeneous coordinate ring of a quadric hypersurface $Q$ in $\P^n$ is a finite graded direct summand of a polynomial ring if and only if $Q$ is a cone over a plane conic.
\end{cor}

With the finite direct summand case thus settled, it is interesting to consider the question of non-finite direct summands. 
The following result is known when the number of variables is even:

\begin{prop}[\cite{HJNB}]
Let $n\geq 4$ be even.
The quadric hypersurface in $n$ variables, which without loss of generality can be taken to be $\C[x_1,\dots,x_n]/(x_1x_2-x_2x_3+\dots\pm x_{n-1}x_n)$, is a direct summand of a regular ring.
\end{prop}

This is an immediate corollary of \cite[Chapter II, Sections~6 and~14]{Weyl}, but we give the details as they may prove useful for the case where the number of variables is odd.

\begin{proof}
Let $c=n/2-1$, and consider the action of $\SL_c(\C)$ on $V^{\oplus c+1}\oplus V^*$. This gives rise to an action of $\SL_c(\C)$ on the symmetric algebra $\C[u_{ij},v_j: 1\leq i \leq n+1, 1\leq j\leq n]$. By \cite{Weyl}, the invariant ring of this action is generated by the $c+1$ maximal minors of the matrix 
$$
\begin{pmatrix}
u_{11} & \cdots & u_{1,c+1}\\
\vdots & \ddots &\vdots
\\
u_{c1} & \cdots & u_{c,c+1}
\end{pmatrix}
$$
and the $c+1$ inner products
$u_{i,1}v_1+\dots+u_{i,c}v_c$.
Write $\Delta_i$ for the minor obtained by removing the $i$-th column and $p_i$ for the inner product $u_{i,1}v_1+\dots+u_{i,c}v_c$.
Then the only relation on these generators is
$$
\Delta_1p_1-\Delta_2p_2+\dots+(-1)^{c+1}\Delta_{c+1} p_{c+1}.
$$
Thus, we have that $\C[x_1,\dots,x_{2n}]/(x_1x_2-x_2x_3+\dots\pm x_{n-1}x_n)$ is isomorphic to the invariant ring of this $\SL_c(\C)$ action (identifying $x_1$ with $\Delta_1$, $x_2$ with $p_1$, and so on up through $x_{n-1}$ with $\Delta_{c+1}$ and $x_n$ with $p_{c+1}$). Since $\SL_c(\C)$ is linearly reductive in characteristic 0, this inclusion splits.
\end{proof}

This result enables us
to obtain the following result for the quadric in five variables, which does not appear in the literature as far as we can tell:

\begin{prop}
The quadric hypersurface in five variables, say $\C[x_1,\dots,x_5]/(x_1x_2+x_3x_4-x_5^2)$, is a direct summand of a regular ring.
\end{prop}

\begin{proof}
From the above, we have an inclusion $$\C[x_1,\dots,x_6]/(x_1x_2+x_3x_4+x_5x_6)\liso \C[\Delta_1,\Delta_2,\Delta_3,p_1,p_2,p_3]\hookrightarrow \C[u_{1,1},\dots,u_{3,2},v_1,v_2]$$ which splits as a direct summand.
Note that we can obtain the ring 
$\C[x_1,\dots,x_5]/(x_1x_2+x_3x_4-x_5^2)$
as a quotient of 
$ \C[x_1,\dots,x_6]/(x_1x_2+x_3x_4+x_5x_6)$ by the ideal $x_5-x_6$, which has image $\Delta_3-p_3$ in $\C[u_{1,1},\dots,u_{3,2},v_1,v_2]$.
Thus, by Lemma~\ref{splitquot} we have that 
$\C[x_1,\dots,x_5]/(x_1x_2+x_3x_4-x_5^2)$ is a direct summand of 
$\C[u_{1,1},\dots,u_{3,2},v_1,v_2]/(\Delta_3-p_3)$.

Now, note that
$$
\Delta_3 = \det\begin{pmatrix} u_{1,1} & u_{1,2}\\u_{2,1} & u_{2,2}\end{pmatrix} =u_{1,1}u_{2,2}-u_{1,2}u_{2,1}
$$
and 
$$
p_3=
u_{3,1}v_{1}
+
u_{3,2}v_{2}
$$
Thus, we have that 
$\C[u_{1,1},\dots,u_{3,2},v_1,v_2]/(\Delta_3-p_3)$ is itself isomorphic to 
$$ \C[x_1,\dots,x_8]/(x_1x_2+\dots+x_7x_8).$$
But by the proposition above, this is itself a direct summand of a regular ring, and thus 
by Lemma~\ref{dirsum2}
so is 
$\C[x_1,\dots,x_5]/(x_1x_2+x_3x_4-x_5^2)$.
\end{proof}

\begin{rem}
As far as we know, it is open whether a quadric in an odd number $n$ of variables is a direct summand of a regular ring for $n\geq 7$.
It is also unclear whether the above results hold in characteristic $p$, since $\SL_c$ will fail to be linearly reductive in characteristic $p$.
\end{rem}

\section{Smooth del Pezzo surfaces}
\label{delPezzo}

For smooth del Pezzo surfaces, we can give a complete description of which homogeneous coordinate rings can be direct summands of regular rings (without the finiteness assumption).

Recall that a del Pezzo surface $X$ is a smooth surface with ample anticanonical bundle $\O_X(-K_X)=\bigwedge^2 T_X$. Del Pezzo surfaces are classified by their degree $d=(-K_X)^2$, and that a del Pezzo surface of degree $d$ is a blow-up of $\P^2$ at $9-d$ general points, except that for degree $8$ it can also be $\P^1\times \P^1$.

In this section, we prove the following:

\begin{thm}
\label{delPezzos}
A complex del Pezzo surface $X_d$ of degree $d$ has a homogeneous coordinate ring that can be recognized as a direct summand of a regular ring if and only if $d\geq 5$.
\end{thm}

\begin{rem}
A del Pezzo surface of degree $X_d$ has a ``natural'' choice of homogeneous coordinate ring $S(X_d,-K_{X_d})=\bigoplus H^0(X_d,\O_{X_d}(-K_{X_d}))$ (although note that $-K_{X_d}$ is very ample only when $d\geq 3$).  
For $d\geq 3$, $-K_{X_d}$ embeds $X_d$ in $\P^{9-d}$.
Theorem~\ref{delPezzos} then says, for example, that:
\begin{itemize}
\item $d=3$: 
If $R=\C[x,y,z,w]/F(x,y,z,w)$, with $F$ a homogeneous cubic with isolated singularity at the origin. Then $R$ cannot be a direct summand of any regular ring.
\item $d=4$:
If $R=\C[x,y,z,w,v]/(Q_1,Q_2)$, with the $Q_i$ homogeneous quadrics, and with $R$ having an isolated singularity at the origin, then 
$R$ cannot be a direct summand of any regular ring.
\end{itemize}
\end{rem}

To prove Theorem~\ref{delPezzos}, we treat three cases: $d\geq 6$, $d\leq 4$, and $d=5$.
First,
note that del Pezzo surfaces of degree $\geq 6$ are toric, and so any homogeneous coordinate ring of such a del Pezzo is a toric ring and thus a direct summand of a regular ring.

For $d\leq 4$, we recall the following theorem:

\begin{thm}[{\cite{Me}}]
If $X$ is a complex del Pezzo surface of degree $d\leq 4$, the tangent bundle $T_X$ is not big, and as a consequence no homogeneous coordinate ring of $X$ can be $D$-simple.
\end{thm}

\begin{cor}
No homogeneous coordinate ring of a del Pezzo of degree $\leq 4$ can be a direct summand of a regular ring.
\end{cor}

The only outstanding case then is the degree-5 del Pezzo, which we handle here:

\begin{thm}
In any characteristic, the homogeneous coordinate ring of the quintic del Pezzo $X_5$ (under any embedding) is a direct summand of the coordinate ring of the Grassmannian $G(2,5)$, and in characteristic 0 thus a direct summand of a regular ring.
\end{thm}

\begin{proof}
We want to consider the Cox ring $R$ of $X_5$. Recall that $X_5$ is the blowup of $\P^2$ at four general points $\set{p_1,\dots,p_4}$. We start by choosing a basis ${H,E_1,\dots,E_4}$ for $\Pic(X_5)\cong \Z^5$, where $H$ is the pullback of the hyperplane class from $\P^2$ and $E_i$ the exceptional divisor over $p_i$.
By results of \cite[Corollary~35]{STV} (proving a conjecture of \cite{BP}), 
the Cox ring $R$ of $X_5$ is isomorphic to $\C[f_{12},f_{13},f_{14},f_{23},f_{24},f_{34},e_1,e_2,e_3,e_4]$
modulo the $4\times 4$ Pfaffians of 
the generic skew-symmetric matrix
$$\begin{pmatrix}
     0&f_{12}&f_{13}&f_{14}&f_{23}\\
     -f_{12}&0&f_{24}&f_{34}&e_1\\
     -f_{13}&-f_{24}&0&e_2&e_3\\
     -f_{14}&-f_{34}&-e_2&0&e_4\\
     -f_{23}&-e_1&-e_3&-e_4&0\\
     \end{pmatrix}.
$$
In this presentation, $f_{ij}$ is the class of a line passing $p_i$ and $p_j$, and $e_i$ is the exceptional divisor over $p_i$.
$R$ is $\Z^{\oplus 5}$-graded: 
for example, $f_{12}$ corresponds to a divisor in the linear series $|\O_{X_5}(H-E_1-E_2)|$, and thus has degree $(1,-1,-1,0,0)$.
$\G_m^5$ acts on $R$ via
$$(\lambda_0,\dots,\lambda_4)\cdot (f_{ij}) = \lambda_0\lambda_i^{-1} \lambda_j^{-1}f_{ij},\quad (\lambda_0,\dots,\lambda_4)\cdot (e_i) = \l_i   e_{i}
,$$ 

We know from the general theory of Cox rings discussed in Section~\ref{cox} that we can recover the various section rings $\bigoplus_m H^0(X_5,\O_X(mD))\subset R$ for $D$ a divisor on $X_5$ as rings of invariants of $\G_m^5$-actions on $R[t]$.
For the sake of making things explicit, we recall how one
recovers the anticanonical embedding corresponding to the divisor $-K_{X_5}$: since $-K_{X_5}=3H-E_1-\dots-E_4$, and thus has degree $(3,-1,-1,-1,-1)$, we extend the $\G_m^5$-action to $R[t]$ by letting 
$$(\lambda_0,\dots,\lambda_4)\cdot (t) = \lambda_0^{-3}\lambda_1^{1}\dots\lambda_4^{1}t.$$
The ring of invariants of $R[t]$ under this action is then isomorphic to the anticanonical section ring $\bigoplus_m H^0(\O_{X_5}(-mK_{X_5})$. The latter is thus generated by  the 
six elements
$$f_{12}f_{34}^2e_3e_4,f_{13}
      f_{23}f_{24}e_2e_3,f_{13}f_{23}
      f_{34}e_3^2,f_{13}f_{24}^2e_2e_4
      ,f_{13}f_{24}f_{34}e_3e_4,f_{14}
      f_{24}f_{34}e_4^2$$
of $H^0(\O_{X_5}(-K_{X_5})$.

Thus, we have that 
$\bigoplus_m H^0(\O_{X_5}(-mK_{X_5})$ is a ring of invariants of $R[t]$ under the linearly reductive (in any characteristic!) group $\G_m^5$, and thus a direct summand of $R[t]$. But in characteristic 0 $R$, and thus $R[t]$, is itself a direct summand a polynomial ring: $R$ is the coordinate ring of the Grassmannian $G(2,5)$, and is thus the invariant ring of the special linear group.  Thus, we are done.
\end{proof}

As a consequence of the proof, we answer a question of \cite{Hara2}:

\begin{cor}
Let $\Char k > 2$.
The homogeneous coordinate ring of a quintic del Pezzo (under any embedding) has the (graded) FFRT property.
\end{cor}

\begin{proof}
\cite[Theorem~1.4]{RSB} implies that the coordinate ring of the Grassmannian $G(2,5)$ under its Pl\"ucker embedding has the graded FFRT property. Thus, the same remains true after adjoining~$t$. Then \cite[Proposition~3.1.6]{SVdB} implies that the graded FFRT property descends to direct summands, so the coordinate ring of a quintic del Pezzo has FFRT as well.
\end{proof}

This also recovers the  result of \cite{Hara1} that the quintic del Pezzo has the global FFRT property (although his approach has the distinct advantage of actually describing the summands of $F_*^e\O_X$, whereas the approach here does not yield this information).

\begin{rem}
The preceding discussion also establishes that the homogeneous coordinate ring of a degree-5 del Pezzo is $D$-simple, in either characteristic 0 or characteristic $p$.
 In characteristic 0, a polynomial ring is obviously $D$-simple, so the above immediately establishes $D$-simplicity of the homogeneous coordinate ring of the degree-5 del Pezzo. 
This in turn implies that the tangent bundle of the degree-5 del Pezzo is big, and
we thus recover one of the results of \cite{HLS}.
In characteristic $p$, the coordinate ring is strongly $F$-regular (for infinitely many $p$ at least), and thus $D$-simple (see \cite[Theorem~2.2.(4)]{Smith}). 
\end{rem}

\begin{rem}
Note 
that we do not obtain directly that the homogeneous coordinate ring of a del Pezzo quintic is a direct summand of a regular ring.
If the homogeneous coordinate ring of $G(2,5)$ were itself known to be a direct summand of a regular ring when $k$ has characteristic $p$, then we would get this. But this is not known, and in particular the known presentation of 
this coordinate ring of a Grassmannian as an invariant ring
has been shown \emph{not} to be pure, and thus not to split; see \cite{HJPS}.
\end{rem}

\begin{rem}
It would be interesting to know what holds true in characteristic $p$.
It will still be the case that homogeneous coordinate rings of toric del Pezzos will be direct summands of regular rings. However, the question of whether the quintic del Pezzo is a direct summand of a regular ring is already subtle: the proof above shows only that it's a direct summand of an invariant ring of a reductive group, and it's unknown whether this can be a direct summand of a regular ring.
Moreover, the homogeneous coordinate rings of del Pezzo surfaces of any degree will be strongly $F$-regular and thus $D$-simple, and so the criteria used for the $d\leq 4$ case above is not useful.
It seems that the expectation is that degree-4 del Pezzos will not have coordinate rings with FFRT, and thus not be direct summands, but as far as we know nothing is known in this direction.
\end{rem}

\section{Smooth hypersurfaces of degree $d\geq 3$ and prime Fano threefolds}
\label{hypersurfaces}
In this section, we discuss the situation for smooth hypersurfaces of degree $d\geq 3$ and prime Fano threefolds (i.e., Fano threefolds with Picard number 1).
The hypersurface case is completely settled, without finiteness or homogeneity assumptions, by combining results of \cite{HLS} and \cite{Hsiao}. Although nothing in this section is original except for the study of the quadric threefold, the conclusions do not appear elsewhere in the literature and so we include them here.

\begin{prop}
Let $X$ be a smooth complex hypersurface of degree $d\geq 3$ in $\P^n$, and let $R=\C[x_0,\dots,x_n]/(F)$ be its homogeneous coordinate ring. $R$ cannot be a direct summand of any regular ring.
\end{prop}

This follows by combining Corollary~\ref{dsimpsum} (due to \cite[Proposition~3.1]{Smith}), which states that direct summands of regular rings are $D$-simple, with
the
following theorems:

\begin{thm}[{\cite[Theorem~1.2]{Hsiao}}]
Let $X$ be a smooth projective variety and $L$ an ample line bundle. If the section ring $\bigoplus H^0(X,L^{\otimes m})$ is $D$-simple, then the tangent bundle $T_X$ is big.
\end{thm}

\begin{thm}[{\cite[Theorem~1.4]{HLS}}]
If $X$ is a smooth hypersurface of degree $d\geq 3$ in $\P^n$, $T_X$ is not big.
\end{thm}

\begin{rem}
Again, the smoothness assumptions here are crucial. We've seen already that (for example) the toric cubic surface $X=V(x^3-yzw)\subset \P^3$ is a finite graded direct summand of a polynomial ring. In fact, one can check also that $H^0(\Sym^3(T_X)(-1))\neq 0$, so that $T_X$ is big.

In fact, generalizing this example yields degree-$d$ hypersurfaces in $\P^n$ for any $d\leq n$ and $n\geq 3$ that are finite direct summands of regular rings: 
we define $X_{n,d}=V(x_0^d-x_1\cdots x_d)\subset \P^n$ for any $d\leq n$, and note that this is a normal  hypersurface (the singularities are codimension-2 quotient singularities). Note moreover that $X_{n,d}$ is the image of $\P^{n-1}$  under the finite morphism $[u_1,\dots,u_{n}]\mapsto [u_1\cdots u_d, u_1^d,\dots,u_n^d]$, and thus the corresponding finite ring inclusion 
$$
\frac{\C[x_0,\dots,x_n]}{x_0^d-x_1\cdots x_d} \hookrightarrow \C[u_1,\dots,u_n]
$$
automatically splits.
One can similarly obtain complete intersection examples by taking the same defining equations in disjoint sets of variables.
\end{rem}

All the examples obtained in this way are toric. This leads to the following:

\begin{quest}
Let $X\subset\P^n$ be a normal hypersurface of degree $d\leq n$. If there is a finite morphism $\P^n\to X$, must $X$ be toric?
\end{quest}

Similarly, by using result from \cite{fanothree}, we can almost completely classify the prime Fano threefolds whose coordinate rings are direct summands of regular rings:

\begin{prop}
Let $X$ be a smooth Fano threefold of Picard number 1. Then the coordinate ring of $X$ is a direct sum of a regular ring if $X$ is $\P^3$ or the quadric threefold, and these are the only cases where this is true, except for possibly the quintic del Pezzo threefold.
\end{prop}

This follows since \cite{fanothree} says the only prime Fano threefolds with big tangent bundle are $\P^3$, the quadric threefold, and the quintic del Pezzo threefold. We've seen in Section~\ref{quad} that the quadric threefold is in fact a direct summand of a regular ring. The only missing case then is that of the quintic del Pezzo threefold, for which it is not clear what to expect.

\section{Coordinate rings of singular varieties}
\label{singcoh}
The preceding sections examined homogeneous rings $R$ for which $\Proj R$ was smooth. In this section, we develop techniques for when $\Proj R$ is singular.
 We begin with a discussion of an alternative approach in the smooth case, before going into an in-depth analysis of the coordinate rings of singular cubics.

\begin{rem}
Let $X$ be a smooth variety and $R=S(X,L)$ the homogeneous coordinate ring of $X$ under some embedding given by an ample line bundle $L$. Let $R\hookrightarrow \C[x_0,\dots,x_n]$ be a module-finite graded inclusion.
As stated before, we thus get a corresponding finite surjection $\P^n\to X$. There are then induced maps $H^i(X;\C)\to H^i(\P^n;\C)$ on singular cohomology. 
Crucially, these maps are \emph{injections}, which follows from the fact that $X$ and $\P^n$ are K\"ahler manifolds (it is not true for an arbitrary surjection of manifolds). 
But $H^i(\P^n;\C)=\C$ for $i$ even and 0 otherwise. So, we must have that $\dim H^i(X;\C)= 1$ for all $i$ even and is 0 for all $i$ odd. This is a very strong condition on $X$, which should be thought of as some kind of ``cohomological minimality''. For example, the only nondegenerate smooth complete intersections in $\P^n$ satisfying this are odd-dimensional quadric hypersurfaces. So, we almost recover the results of Section~\ref{finsum} on complete intersections, modulo this case. Likewise, there are other smooth varieties with the same cohomology as $\P^n$, but they are somewhat rare.
\end{rem}

The utility of the above approach, despite the fact that it is not as strong in the smooth case as the methods of Section~\ref{finsum}, is that it can be generalized outside the smooth case.

As an example, we will show the following theorem:

\begin{thm}
\label{singdel}
The homogeneous coordinate ring of a singular complex cubic surface is a finite graded direct summand of a regular ring if and only if the surface has quotient singularities of type $3A_2$.
\end{thm}

This mostly recovers a special case of a theorem of \cite[Lemma~10]{Gurjar} classifying which log del Pezzos (of any degree) are dominated by $\P^2$, though at this point we do rely on that paper to rule out the two cases of cubic surfaces with $A_1A_5$ or $E_6$ singularities.

To prove this theorem, we'll examine the cohomological contributions of the singularities of a cubic surface. We start with some general lemmas:

\begin{lem}
Let $X$ be a complex variety with quotient singularities, and let $\pi:\wtilde X\to X$ be a resolution of singularities. Then the induced morphisms $H^i(X;\C)\to H^i(\wtilde X;\C)$ are injective.
\end{lem}

\begin{proof}
By \cite[Corollary~8.2.22]{HTT}, because $X$ has quotient singularities, $H^i(X;\C)$ is isomorphic to the $i$-th intersection cohomology $IH^i(X)$. But for any projective variety $X$ and resolution $\wtilde X\to X$, $IH^i(X)$ is a direct summand of $H^i(\wtilde X;\C)$ by \cite[Corollary~8.2.27]{HTT}, and in particular injects into it. Thus, we have that $H^i(X;\C)\cong IH^i(X)\to H^i(\wtilde X;\C)$ is injective.
\end{proof} 
 
\begin{lem}
\label{inj}
If $X$ is a complex projective variety with quotient singularities, and $f:Y\to X$ is any surjective morphism from a projective variety $Y$, 
the induced morphisms
$H^i(X;\C)\to H^i(Y;\C)$ are injective.
\end{lem}

\begin{proof}
Let
$\pi:\wtilde X\to X$ be a resolution of singularities,
and let
 $Z$ be any smooth projective variety with a morphism to $Y\times _X \wtilde X$ and surjecting onto $\wtilde X$ (e.g., a resolution of singularities of an irreducible component of $Y\times_X \wtilde X$ dominating $\wtilde X$). The commutative diagram 
$$
\begin{tikzcd}
Z\ar[r]\ar[d] & \wtilde X\ar[d,"\pi"]\\
Y \ar[r,"f"']& X
\end{tikzcd}
$$
gives rise to a commutative diagram
$$
\begin{tikzcd}
H^i(Z;\C)&  H^i(\wtilde X;\C)\ar[l,hook']\\
H^i(Y;\C) \ar[u]& H^i(X;\C)\ar[u,hook]\ar[l]
\end{tikzcd}
$$
Since $X$ has quotient singularities, $H^i(X;\C)\to H^i(\wtilde X;\C)$ is injective, and $H^i(\wtilde X;\C)\to H^i(Z;\C)$ is injective because $Z\to \wtilde X$ is a surjective morphism of smooth projective varieties.
Thus, $H^i(X;\C)\to H^i(Y;\C)$ must be injective as well.
\end{proof}

Now,
let $X$ be a projective \emph{surface} and $L$ an ample line bundle. We want to consider when the homogeneous coordinate ring $R=\bigoplus H^0(X,L^{\otimes m})$ can be a finite direct summand.
We may then assume that $X$ is normal, and thus has isolated singularities.
By Lemma~\ref{klt}, if $R$ is a direct summand of a polynomial ring, the singularities of $\Spec R$ are klt. We must then have that the singularities of 
$X$ itself are klt (and that $X$ is Fano and $L$ a multiple of the anticanonical divisor).
Thus, by well-known results on surface singularities, the singularities of $X$ are in fact quotient.

\begin{lem}
\label{numres}
Let $\P^n \to X$ be a surjective morphism, and assume $X$ is a surface with quotient singularities. If $\pi:\wtilde X\to X$ is the minimal resolution of singularities, then
$$|\set{E: \text{$E$ is a $\pi$-exceptional irreducible divisor}}| = \dim H^2(\wtilde X;\C)-1,$$
\end{lem}


\begin{proof}

Let $\Sigma = X_\sing$. By \cite[Corollary-Definition~5.37]{PS}, we have a long exact sequence
$$
\cdots\to H^1(E) \to H^2(X)\to H^2(\wtilde X)\oplus H^2(\Sigma) \to H^2(E)\to H^3(X)\to \cdots
$$
(in fact, this is even a long exact sequence of mixed Hodge structures, though we will not need this).
Since $\dim \Sigma = 0$, $H^2(\Sigma)=0$. Since there is a surjection $\P^2\to X$ and $X$ is assumed to have quotient singularities, $H^2(X)=1$ and $H^3(X)=0$.

Since $X$ has quotient singularities, it has rational singularities, and thus the exceptional locus of the minimal resolution $\wtilde X$ has simple normal crossings (see, e.g., the comment following Theorem~1.10 in \cite{Stevens}); moreover, 
the exceptional locus $E$ is the union of rational curves intersecting transversely, and with no cycles in their intersection graph.
Say there are $r$ such exceptional divisors.
It's easily checked via Mayer--Vietoris that $H^1(E)=0$ and $H^2(E)=r$.
(Note that the former claim depends on having no cycles in the intersection graph of the components of $E$: if $E$ is three $\P^1$'s forming a triangle, one can check that $H^1(E)=\C$.)

The long exact sequence thus becomes
$$
0 \to \C \to H^2(\wtilde X) \to H^2(E)\to 0,
$$
and thus $\dim H^2(\wtilde X)-1$ is exactly $\dim H^2(E)$, which is the number of exceptional divisors in the minimal resolution $\pi:\wtilde X\to X$.
\end{proof}

Recall that for a 
\emph{Gorenstein} quotient surface singularity (which is necessarily a hypersurface singularity),
the
number of exceptional components over $p\in X_\sing$ is the Milnor number $\mu(p)$. We immediately have the following corollary:

\begin{cor}
If $\P^n\to X$ is a surjective morphism, with $X$ a surface with only Du Val singularities (i.e., Gorenstein quotient singularities), if $\wtilde X \to X$ is the minimal resolution of $X$, then
$$\sum_{p\in X_\sing }\mu(p) = \dim H^2(\wtilde X;\C)-1,$$
where $\mu(p)$ is the Milnor number of $p$.
\end{cor}


With these preliminaries, we can now prove Theorem~\ref{singdel}:

\begin{proof}[Proof of Theorem~\ref{singdel}]
We now specialize to the case where $X$ is a cubic surface with quotient singularities. The minimal resolution $\wtilde X\to X$ is then the blowup of $\P^2$ at six points, and so $H^2(\wtilde X;\C)=\C^7$.  The possible configurations of quotient singularities of $X$ are known (see, for example, \cite[Table~2]{cubics}), and are exactly the following:
$$
\displaylines{
A_1, 2A_1, A_1A_2, 3A_1, A_1A_3, 2A_1A_2, 4A_1, A_1A_4,2A_1A_3,A_12A_2,
\hfill\cr\hfill
A_1A_5,A_2,2A_2,3A_2,A_3,A_4,A_5,D_4(1),D_4(2),D_5,E_6
}
$$
Only three of these have sum of Milnor numbers equal to 6: $3A_2$, $A_1A_5$, and $E_6$.
\begin{itemize}
\item 
The $3A_2$ case is exactly the toric surface
$$
\C[x,y,z,w]/(x^3-yzw),
$$
which we've already seen is a finite direct sum of a polynomial ring.
\item The $A_1A_5$ surface is defined by
$$
\C[x,y,z,w]/(y^3+w(x^2+yz)).
$$
\item The $E_6$ surface is defined by
$$\C[x,y,z,w]/(y^{3}+w(x^{2}+zw))$$
\end{itemize}
For the latter two, \cite[Lemma~10]{Gurjar} has shown that they do not admit any morphisms (necessarily finite) from $\P^2$, using different methods from the ones here. 
\end{proof}

It would be interesting to have a direct, elementary proof of the $A_1A_5$ and $E_6$ cases that does not make use of the classification of \cite[Lemma~10]{Gurjar}, though we have not found one yet.

\begin{rem}
The $3A_2$ case appears also in the study of the moduli space of cubic surfaces: it is the unique strictly semistable point  in the moduli space of semistable cubic surfaces, which is a one-point compactification of the space of stable cubic surfaces
(see, e.g., \cite[Theorem~7.24]{Mukai}).
It would be interesting to know if there's a conceptual reason why the only cubic surface that is a finite direct summand of a polynomial ring is also the unique strictly semistable cubic surface, and if this occurs in other dimensions. For example, is the singular quartic $V(x^4-yzwu)$ (which is a finite direct summand of a polynomial ring) a strictly semistable quartic threefold?
\end{rem}

\begin{rem}
One could repeat the same analysis with other del Pezzo surfaces with quotient singularities, and attempt to determine which of them are finite direct summands.
Again, if $X_d$ is a del Pezzo of degree $d\leq 5$ with quotient singularities, the minimal resolution $\wtilde X$ will be the blowup of $\P^2$ at $9-d$ general points, and thus will have $\dim H^2(\wtilde X) = 1+(9-d)$. A necessary condition to be a finite direct summand is thus that the sum of the Milnor numbers at the singular points is $9-d$.
As an example, consider degree-4 del Pezzos, which are intersections of two quadrics in $\C[x,y,z,w,u]$.
The degree-4 del Pezzo defined by
$$
\C[x,y,z,w,u]/(w^{2}-yu, x^{2}-zw)
$$
is toric, and a finite direct summand. It has one $A_3$ singularity and two $A_1$ singularities, so the sum of the Milnor numbers is 5, as we know is necessary.

We note also that the finiteness condition is critical in order to use the cohomological obstruction:
the ring
$$
\C[x,y,z,w,u]/(yz-wu, x^{2}-wu)
$$
is also toric, and thus a direct summand of a polynomial ring, but it cannot be a finite direct summand, as it has only four $A_1$ singularities, and thus the sum of the Milnor numbers is 4.

In some sense, the cubic surface case is especially clear, because there is only one (normal) toric cubic surface up to isomorphism, while (for example)  we see above that there are at least two toric degree-4 del Pezzos, with different behavior.
\end{rem}

\begin{rem}
After the initial calculations above were completed, we discovered that \cite[Lemma~10]{Gurjar} had completely classified the Gorenstein log del Pezzo surfaces dominated by $\P^2$, using different methods from those above.
This involves showing that almost all Gorenstein log del Pezzo surfaces dominated by $\P^2$ are in fact finite quotients of $\P^2$. 
As a corollary, if such Gorenstein log del Pezzo surface admits a finite covering by $\P^2$, it admits a finite covering by $\P^2$ unramified in codimension 1.
In general, it would be interesting to know to what degree one can ``remove ramification'' from a given finite cover by $\P^n$.
\end{rem}

\begin{rem}
There are several obstacles to extending the above analysis in higher dimensions.

The first obstacle is that in higher dimensions, klt singularities are not necessarily quotient singularities, and thus the morphisms on cohomology induced by a surjection from a smooth projective variety need not be injective. It might be possible to obtain injectivity in certain degrees, or on certain graded pieces of the mixed Hodge structure, and thus still be able to utilize parts of the above approach.

The second is that the relation between the cohomology of $X$ and a resolution of singularities becomes more complicated, even when $X$ has quotient singularities. Because the singularities of $X$ need not be isolated, and the intersection of the exceptional divisors of a resolution will not be zero-dimensional, understanding the cohomological relation between $X$ and its resolution will require a more intricate and sophisticated analysis.
\end{rem}

\begin{rem}
One can also attempt to weaken the restriction on the singularities of $X$. The important feature of quotient singularities is the intersection cohomology is the same as the singular cohomology. This is true for a broader class of varieties: the \emph{rationally smooth} varieties (see, e.g., \cite[Definition~8.20]{HTT}). Quotient singularities are rationally smooth, but the converse is not true. For example, \cite[Lemma in Section~1.3]{Brion} implies that ordinary double point singularity of dimension $n$, analytically locally given by 
$ x_0^2+\dots+x_n^2 $,
is rationally smooth if and only if $n$ is odd. In contrast, our results in Section~\ref{finite} imply that this is a quotient singularity only when $n=3$.
Note that this shows that rationally smooth singularities are still a fairly restrictive notion, as for example threefolds with ordinary double points will not be rationally smooth, thus precluding the above approach.
\end{rem}

\let\l\oldl
\let\b\oldb
\bibliographystyle{alpha}
\bibliography{link}
\end{document}